\documentclass[10pt,oneside,leqno]{amsart}
\usepackage{amssymb}
\usepackage{amsfonts}
\usepackage{graphicx}
\usepackage{amsmath}
\usepackage{color}
\usepackage{lipsum}
\usepackage{epstopdf}
\usepackage{algorithmic}
\usepackage{appendix}
\usepackage{amsthm}
\usepackage{amsmath}
\usepackage{amssymb}
\usepackage{amsfonts}
\usepackage{amscd}
\usepackage{amsfonts}
\usepackage{amsthm}
\usepackage{graphicx}
\usepackage[T1]{fontenc}
\usepackage[utf8]{inputenc}
\usepackage[french,english]{babel}
\usepackage{color}
\usepackage{marvosym}
\usepackage{amsthm}
\usepackage[export]{adjustbox}
\usepackage{mathtools}

\DeclareMathOperator{\diag}{diag}
\newtheorem{theorem}{Theorem}

\newtheorem{definition}{Definition}

\newtheorem{proposition}{Proposition}

\newcommand{\func}[1]{\operatorname{#1}}

\begin{document}

\title[]{Exponentially convergent trapezoidal rules to approximate
fractional powers of operators}
\thanks{This work was partially supported by GNCS-INdAM and FRA-University
of Trieste. The authors are members of the INdAM research group GNCS}
\author{Lidia Aceto}
\address{Lidia Aceto\\
Universit\`{a} del Piemonte Orientale\\
Dipartimento di Scienze e Innovazione Tecnologica, viale Teresa Michel 11,
15121 Alessandria\\
Italy}
\email{lidia.aceto@uniupo.it}
\author{Paolo Novati}
\address{Paolo Novati \\
Universit\`{a} di Trieste\\
Dipartimento di Matematica e Geoscienze, via Valerio 12/1, 34127 Trieste\\
Italy}
\email{novati@units.it}

\begin{abstract}
  In this paper we are interested in the approximation of fractional powers of
self-adjoint positive operators. Starting from the integral representation
of the operators, we apply the trapezoidal rule combined with a
single-exponential and a double-exponential transform of the integrand
function. For the first approach our aim is only to review some theoretical
aspects in order to refine the choice of the parameters that allow a faster
convergence. As for the double exponential transform, in this work we show
how to improve the existing error estimates for the scalar case and also
extend the analysis to operators. We report some numerical experiments to
show the reliability of the estimates obtained.
\end{abstract}



\maketitle

\section{Introduction}
In this work we are interested in the numerical approximation of ${\mathcal{L}}^{-\alpha }$, $\alpha \in (0,1).$ Here ${\mathcal{L}}$ is a self-adjoint positive operator acting in an Hilbert space ${\mathcal{H}}$ in which the eigenfunctions of ${\mathcal{L}}$ form an orthonormal basis of ${\mathcal{H}},$ so that ${\mathcal{L}}^{-\alpha }$ can be written through the spectral
decomposition of ${\mathcal{L}}$. In other words, for a given $g\in {\mathcal{H}}$, we have
\begin{equation}
{\mathcal{L}}^{-\alpha }g=\sum_{j=1}^{+\infty }\mu _{j}^{-\alpha }\langle
g,\varphi _{j}\rangle \varphi _{j}  \label{dec}
\end{equation}
where $\mu _{j}$ and $\varphi _{j}$ are the eigenvalues and the eigenfunctions of ${\mathcal{L}},$ respectively, and $\langle \cdot ,\cdot \rangle $ denotes the ${\mathcal{H}}$-inner product. Throughout the paper we also assume $\sigma (\mathcal{L})\subseteq \lbrack 1,+\infty )$, where $\sigma (\mathcal{L})$ denotes the spectrum of ${\mathcal{L}}$.

Applications of (\ref{dec}) include the numerical solution of fractional
equations involving the anomalous diffusion, in which ${\mathcal{L}}$ is
related to the Laplacian operator, and this is the main reason for which in
recent years a lot of attention has been placed on the efficient
approximation of fractional powers. Among the approaches recently introduced
we quote here the methods based on the best uniform rational approximations
of functions closely related to $\lambda ^{-\alpha }$ that have been studied
in \cite{HLMMV, HLMMP, HML, HM}. Another class of methods relies on
quadrature rules arising from the Dunford-Taylor integral representation of $%
\lambda ^{-\alpha }$ \cite{ABDN,AN0,AN,ANI,Bo,V19,V18}.  Very recently, time
stepping methods for a parabolic reformulation of fractional diffusion
equations,  proposed  in \cite{V15},   have been   interpreted
by Hofreither in \cite{H} as rational approximations of $\lambda ^{-\alpha }.$

In this work, starting from the integral representation
\begin{equation}
\mathcal{L}^{-\alpha }=\frac{2\sin (\alpha \pi )}{\pi }\int_{0}^{+\infty
}t^{2\alpha -1}({\mathcal{I}}+t^{2}\mathcal{L})^{-1}dt,\qquad \alpha \in
(0,1),  \label{intOR}
\end{equation}%
where ${\mathcal{I}}$ is the identity operator in ${\mathcal{H}}$, we
consider the trapezoidal rule applied to the single and the double-exponential transform of the integrand function. The former approach has
been extensively studied in \cite{Bo}, where the authors also provide
reliable error estimates. The rate of convergence has been shown to be of
type%
\begin{equation} \label{exr}
\exp (-c\sqrt{n}),
\end{equation}
where $n$ is closely related to the number of nodes. Our aim here is just to
review some theoretical aspects in order to refine the choice of the
parameters that allow faster convergence, even if, still of type (\ref{exr}). 
As for the double-exponential transform, widely investigated in \cite%
{OK,OK1,M,TSMM, Tref} for general scalar functions, in this work we show how
to improve the existing error estimates for the function $\lambda ^{-\alpha }
$. We also extend the analysis to operators, showing that it is possible to
reach a convergence rate of type%
\begin{equation*}
\exp \left( -c\sqrt{\frac{n}{\ln n}}\right) .
\end{equation*}%
While theoretically disadvantageous with respect to the single-exponential
approach, we show that the double-exponential approach is actually faster at
least for $\alpha \in \lbrack 1/2,1)$.

The paper is organized as follows. In Section \ref{sec2} we make a short
background concerning the trapezoidal rule with particular attention to
functions that decay exponentially at infinity. In Section \ref{sec3} we
review the existing convergence analysis for the trapezoidal rule combined
with a single-exponential transform and we refine the choice of the
parameters that allow a faster convergence.  Section \ref{sec4} is devoted to
the trapezoidal rule combined with a double-exponential transform. Here the
convergence analysis is derived for the approximation of the scalar function
$\lambda ^{-\alpha }$ and is then extended to the case of the operator ${%
\mathcal{L}}^{-\alpha }.$ Some concluding remarks are finally reported in
Section \ref{sec5}.


\section{A general convergence result for the trapezoidal rule}

\label{sec2}

Given a generic continuous function $f:\mathbb{R}\rightarrow
\mathbb{R}$, in this section we make a short background concerning the
trapezoidal approximation
\begin{equation*}
I(f)=\int_{-\infty }^{+\infty }f(x)dx\approx h\sum_{\ell =-\infty }^{+\infty
}f(\ell h),
\end{equation*}%
where $h$ is a suitable positive value. Given $M$ and $N$ positive integers,
we denote the truncated trapezoidal rule by
\begin{equation*}
T_{M,N,h}(f)=h\sum_{\ell =-M}^{N}f(\ell h).
\end{equation*}%
For the error we have
\begin{equation}
\mathcal{E}_{M,N,h}(f):=\left\vert I(f)-T_{M,N,h}(f)\right\vert \leq
\mathcal{E}_{D}+\mathcal{E}_{T_{L}}+\mathcal{E}_{T_{R}},  \notag
\label{3pezzi}
\end{equation}%
where
\begin{align*}
\mathcal{E}_{D}& =\left\vert \int_{-\infty }^{+\infty }f(x)dx-h\sum_{\ell
=-\infty }^{+\infty }f(\ell h)\right\vert , \\
\mathcal{E}_{T_{L}}& =h\sum_{\ell =-\infty }^{-M-1}\left\vert f(\ell
h)\right\vert ,\qquad \mathcal{E}_{T_{R}}=h\sum_{\ell =N+1}^{+\infty
}\left\vert f(\ell h)\right\vert .
\end{align*}%
The quantities $\mathcal{E}_{D}$ and $\mathcal{E}_{T}:=\mathcal{E}_{T_{L}}+%
\mathcal{E}_{T_{R}}$ are referred to as the discretization error and the
truncation error, respectively.

\begin{definition}
\cite[Definition 2.12]{LB} \label{def1} For $d>0,$ let $\mathcal{D}_{d}$ be
the infinite strip domain of width $2d$ given by
\begin{equation*}
\mathcal{D}_{d}=\{\zeta \in \mathbb{C}:\,|\func{Im}(\zeta )|<d\}.
\end{equation*}%
Let $B(\mathcal{D}_{d})$ be the set of functions analytic in $\mathcal{D}%
_{d} $ that satisfy
\begin{equation*}
\int_{-d}^{d}|f(x+i\eta )|d\eta =\mathcal{O}(|x|^{a}),\quad x\rightarrow \pm
\infty ,\quad 0\leq a<1,
\end{equation*}%
and
\begin{equation*}
\mathcal{N}(f,d)=\lim_{\eta \rightarrow d^{-}}\left\{ \int_{-\infty
}^{+\infty }|f(x+i\eta )|dx+\int_{-\infty }^{+\infty }|f(x-i\eta
)|dx\right\} <+\infty .
\end{equation*}
\end{definition}

The next theorem gives an estimate for the discretization error of the
trapezoidal rule when applied to functions in $B(\mathcal{D}_{d})$.

\begin{theorem}
\label{teo1} \cite[Theorem 2.20]{LB} Assume $f\in B(\mathcal{D}_{d})$. Then%
\begin{equation}
\mathcal{E}_{D}\leq \frac{\mathcal{N}(f,d)}{2\sinh (\pi d/h)}e^{-\pi d/h}.
\label{Edisc}
\end{equation}
\end{theorem}

\begin{theorem}
\label{teo2} Assume $f\in B(\mathcal{D}_{d})$ and that there are positive
constants $\beta ,\gamma $ and $C$ such that
\begin{equation}
|f(x)|\leq C\left\{
\begin{array}{cc}
\exp (-\beta |x|), & x<0, \\
\exp (-\gamma |x|), & x\geq 0.%
\end{array}%
\right.  \label{boundfSE}
\end{equation}%
Then,
\begin{equation}
\mathcal{E}_{M,N,h}(f)\leq \frac{N(f,d)}{2\sinh (\pi d/h)}e^{-\pi d/h}+\frac{%
C}{\beta }e^{-\beta Mh}+\frac{C}{\gamma }e^{-\gamma Nh}.  \label{bb}
\end{equation}
\end{theorem}

\begin{proof}
By (\ref{boundfSE}) we immediately have
\begin{equation*}
\mathcal{E}_{T_{L}}\leq \frac{C}{\beta }e^{-\beta Mh},\quad \mathcal{E}%
_{T_{R}}\leq \frac{C}{\gamma }e^{-\gamma Nh}.
\end{equation*}%
Using Theorem \ref{teo1} we obtain (\ref{bb}).
\end{proof}

The above result states that for functions that decay exponentially for $%
x\rightarrow \pm \infty $ it may be possible to have exponential convergence
after a proper selection of $h$. When working with the more general situation%
\begin{equation}
I(g):=\int_{a}^{b}g(t)dt,  \label{int1}
\end{equation}%
one can consider a suitable conformal map
\begin{equation*}
\psi :(-\infty ,+\infty )\rightarrow (a,b),
\end{equation*}%
and, through the change of variable $t=\psi (x)$, transform (\ref{int1}) to
\begin{equation*}
I(g_{\psi }):=\int_{-\infty }^{+\infty }g_{\psi }(x)dx,\qquad g_{\psi
}(x)=g(\psi (x))\psi ^{\prime }(x).
\end{equation*}%
A suitable choice of the mapping $\psi $ may allow to work with a function $%
g_{\psi }$ that fulfills the hypothesis of Theorem \ref{teo2} so that $I(g)$
can be evaluated with an error that decays exponentially.

Since the aim of the paper is the computation of $\mathcal{L}^{-\alpha }$
with $\sigma (\mathcal{L})\subseteq \lbrack 1,+\infty ),$ for $\lambda \geq
1 $ we consider now the integral representation (\ref{intOR})
\begin{equation}
\lambda ^{-\alpha }=\frac{2\sin (\alpha \pi )}{\pi }\int_{0}^{+\infty
}t^{2\alpha -1}(1+t^{2}\lambda )^{-1}dt,\qquad \alpha \in (0,1).
\label{orig}
\end{equation}%
Defining
\begin{equation}
g_{\lambda }(t):=t^{2\alpha -1}(1+t^{2}\lambda )^{-1},  \label{glam}
\end{equation}%
and a change of variable $t=\psi (x)$, $\psi :(-\infty ,+\infty )\rightarrow
(0,+\infty )$, let%
\begin{equation}  \label{NUM}
g_{\lambda ,\psi }(x)=g_{\lambda }(\psi (x))\psi ^{\prime }(x).
\end{equation}%
Let moreover%
\begin{equation*}
\mathcal{Q}_{M,N,h}^{\alpha }(g_{\lambda ,\psi })=\frac{2\sin (\alpha \pi )}{%
\pi }h\sum_{\ell =-M}^{N}g_{\lambda ,\psi }(\ell h)
\end{equation*}%
be the truncated trapezoidal rule for the computation of $\lambda ^{-\alpha
} $, that is, for the computation of%
\begin{equation*}
\frac{2\sin (\alpha \pi )}{\pi }\int_{0}^{+\infty }g_{\lambda }(t)dt=\frac{%
2\sin (\alpha \pi )}{\pi }\int_{-\infty }^{+\infty }g_{\lambda ,\psi }(x)dx.
\end{equation*}%
We denote the error by
\begin{align}
E_{M,N,h}(\lambda ) &=\left\vert \lambda ^{-\alpha }-\mathcal{Q}%
_{M,N,h}^{\alpha }(g_{\lambda ,\psi })\right\vert  \notag \\
&=\frac{2\sin (\alpha \pi )}{\pi }\mathcal{E}_{M,N,h}(g_{\lambda ,\psi }),
\label{erod}
\end{align}%
and for operator argument%
\begin{equation}
E_{M,N,h}(\mathcal{L})=\left\Vert {\mathcal{L}}^{-\alpha }-\mathcal{Q}%
_{M,N,h}^{\alpha }(g_{{\mathcal{L}},\psi })\right\Vert _{\mathcal{H}%
\rightarrow \mathcal{H}}.  \label{erop}
\end{equation}%
The remainder of the paper is devoted to the analysis of two special choices
for $\psi $, the single-exponential (SE) and the double-exponential (DE)
transforms.


\section{The Single-Exponential transform}

\label{sec3}

The SE transform is defined by%
\begin{equation}
\psi _{SE}(x)=\exp (x),  \label{SE3}
\end{equation}%
so that from (\ref{glam}) and (\ref{NUM}) we get
\begin{equation*}
g_{\lambda ,\psi _{SE}}(x)=e^{2\alpha x}(1+e^{2x}\lambda )^{-1}.
\end{equation*}%
Since the poles of this function are given by
\begin{equation*}
x_{k}=-\frac{1}{2}\log \lambda -i(2k+1)\frac{\pi }{2},\qquad k\in \mathbb{Z},
\end{equation*}%
we have that $g_{\lambda ,\psi _{SE}}$ is analytic in
\begin{equation*}
\mathcal{D}_{\psi _{SE}}=\left\{ \zeta \in \mathbb{C}:\,|\func{Im}(\zeta )|<%
\frac{\pi }{2}\right\} ,
\end{equation*}%
that is, the strip domain with $d=\pi /2,$ independently of $\alpha $ and $%
\lambda .$ Now, in order to prove that $g_{\lambda ,\psi _{SE}}$ belongs to $%
B(\mathcal{D}_{\psi _{SE}})$ (see Definition \ref{def1}), following the
analysis given in \cite{Bo} we first note that for $\eta \in \mathbb{R},$ $%
|\eta |<\pi /2$ and $\lambda \geq 1,$
\begin{equation*}
\left\vert (1+e^{2(x+i\eta )}\lambda )^{-1}\right\vert \leq \left\{
\begin{array}{cl}
1, & x<0, \\
e^{-2x}, & x\geq 0.%
\end{array}%
\right.
\end{equation*}%
Therefore,
\begin{equation}
\left\vert g_{\lambda ,\psi _{SE}}(x+i\eta )\right\vert \leq \left\{
\begin{array}{ll}
e^{2\alpha x}, & \mbox{for }x<0, \\
e^{-2(1-\alpha )x}, & \mbox{for }x\geq 0.%
\end{array}%
\right.  \label{st}
\end{equation}%
This implies that
\begin{align*}
\mathcal{N}\left( g_{\lambda ,\psi _{SE}},\pi /2\right) & =\lim_{\eta
\rightarrow (\pi /2)^{-}}\left\{ \int_{-\infty }^{+\infty }\left\vert
g_{\lambda ,\psi _{SE}}(x+i\eta )\right\vert dx+\int_{-\infty }^{+\infty
}\left\vert g_{\lambda ,\psi _{SE}}(x-i\eta )\right\vert dx\right\} \\
& \leq \frac{1}{\alpha (1-\alpha )},
\end{align*}%
and also that
\begin{equation*}
\int_{-\pi /2}^{\pi /2}|g_{\lambda ,\psi _{SE}}(x+i\eta )|d\eta =\mathcal{O}%
(1),\quad x\rightarrow +\infty .
\end{equation*}%
Therefore, by (\ref{def1}) we can conclude that $g_{\lambda ,\psi
_{SE}}$ belongs to $B(\mathcal{D}_{\psi _{SE}}).$ By (\ref{Edisc}), for the
discretization error we have
\begin{equation*}
\left\vert \int_{-\infty }^{+\infty }g_{\lambda ,\psi
_{SE}}(x)dx-h\sum_{\ell =-\infty }^{+\infty }g_{\lambda ,\psi _{SE}}(\ell
h)\right\vert \leq \frac{1}{\alpha (1-\alpha )}\frac{e^{-\pi d/h}}{2\sinh
(\pi d/h)},\quad d=\pi /2.
\end{equation*}

Since%
\begin{equation*}
\frac{e^{-t}}{2\sinh (t)}=e^{-2t}\left( 1+\mathcal{O}(e^{-2t})\right) \quad
\text{as }t\rightarrow +\infty ,
\end{equation*}%
by (\ref{teo2}), for $h\leq 2\pi d$ we obtain
\begin{align}
\mathcal{E}_{M,N,h}(g_{\lambda ,\psi _{SE}})& =\left\vert \int_{-\infty
}^{+\infty }g_{\lambda ,\psi _{SE}}(x)dx-h\sum_{\ell =-M}^{N}g_{\lambda
,\psi _{SE}}(\ell h)\right\vert  \notag \\
& \leq \frac{1}{\alpha (1-\alpha )}e^{-2\pi d/h}+\frac{1}{2\alpha }%
e^{-2\alpha Mh}+\frac{1}{2(1-\alpha )}e^{-2(1-\alpha )Nh}.  \label{trexp}
\end{align}%
After choosing $h$, the contribute of the three exponentials can be equalized
by taking $M$ and $N$ such that%
\begin{equation*}
\pi d/h\approx \alpha Mh\approx (1-\alpha )Nh,
\end{equation*}%
that is
\begin{equation*}
M=\left\lceil \frac{\pi d}{\alpha h^{2}}\right\rceil ,\quad N=\left\lceil
\frac{\pi d}{(1-\alpha )h^{2}}\right\rceil ,
\end{equation*}%
where $\left\lceil \cdot \right\rceil $ is the ceiling function. Denoting by $%
n=M+N+1$ the total number of inversions we have that%
\begin{equation*}
n\approx \frac{\pi d}{h^{2}}\frac{1}{\alpha (1-\alpha )},
\end{equation*}%
and therefore by (\ref{trexp}) we obtain%
\begin{equation*}
\mathcal{E}_{M,N,h}(g_{\lambda ,\psi _{SE}})\leq \frac{3}{2}\frac{1}{\alpha
(1-\alpha )}\exp \left( -2\sqrt{\pi d\alpha (1-\alpha )}\sqrt{n}\right) .
\end{equation*}%
By (\ref{erop}), since $\mathcal{L}$ is assumed to be self-adjoint and $%
\sigma (\mathcal{L})\subseteq \lbrack 1,+\infty )$, we have that%
\begin{equation*}
E_{M,N,h}(\mathcal{L})\leq \frac{2\sin (\alpha \pi )}{\pi }\max_{\lambda
\geq 1}\mathcal{E}_{M,N,h}(g_{\lambda ,\psi _{SE}}),
\end{equation*}%
and then, finally, taking $d=\pi /2$ we obtain
\begin{equation}
E_{M,N,h}(\mathcal{L})\leq \frac{\sin (\alpha \pi )}{\pi }\frac{3}{\alpha
(1-\alpha )}\exp \left( -\pi \sqrt{2\alpha (1-\alpha )}\sqrt{n}\right)
,\quad n=M+N+1.  \label{stima}
\end{equation}%
The analysis just presented is almost identical to the one given in \cite{Bo}. 
Nevertheless in that paper the authors define $d=\pi /4$, while here we
have shown that one can take $d=\pi /2$. This choice produces a remarkable
speedup, as shown in  Figure \ref{Figure1}, where for different values of $%
\alpha $\ we have considered the error in the computation of $\mathcal{L}%
^{-\alpha }$ for the artificial operator
\begin{equation}
\mathcal{L}=\left[ \diag(1, 2,\dots,100)\right] ^{8},\quad \sigma (\mathcal{L}%
)\subseteq \lbrack 1,10^{16}]\text{.}  \label{scex}
\end{equation}

\begin{figure}[htbp]
  \centering
\includegraphics[width=0.95\textwidth]{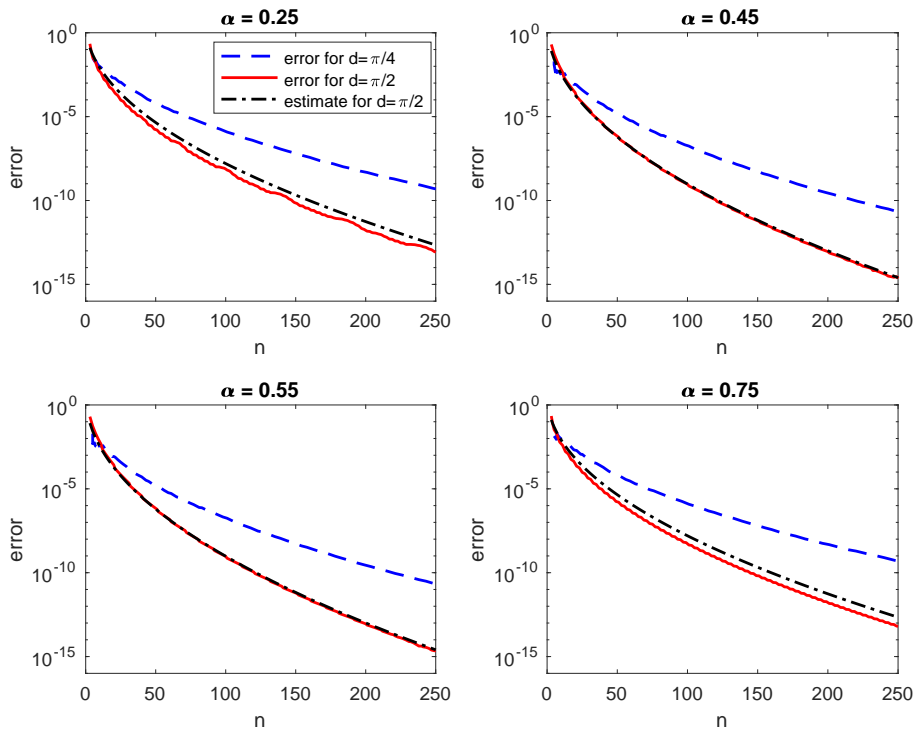}
\caption{Error for the trapezoidal rule applied with the single-exponential
transform with $d= \pi /4$ and $d=\pi /2$, and error estimate
given by (\ref{stima}). }
\label{Figure1}
\end{figure}


\section{Double-Exponential transformation}
\label{sec4}

The DE transform we use here is given by
\begin{equation}
\psi _{DE}(x)=\exp \left( \frac{\pi }{2}\sinh (x)\right) .  \label{DE3}
\end{equation}%
We consider in (\ref{orig}) the change of variable
\begin{equation*}
\tau t^{2}=\left( \psi _{DE}(x)\right) ^{2}=\exp (\pi \sinh (x)),\qquad \tau
>0.
\end{equation*}%
The function involved in this case is
\begin{align}
g_{\lambda ,\psi _{DE}}(x) &=\frac{\pi }{2}\tau ^{1-\alpha }\frac{\exp
(\alpha \pi \sinh (x))}{\tau +\lambda \exp (\pi \sinh (x))}\cosh (x)  \notag
\\
&=\frac{\pi }{2}\lambda ^{-\alpha }\frac{\left( \lambda /\tau \exp (\pi
\sinh (x))\right) ^{\alpha }}{1+\lambda /\tau \exp (\pi \sinh (x))}\cosh (x),
\label{ide}
\end{align}%
and we employ the trapezoidal rule to compute%
\begin{equation*}
\lambda ^{-\alpha }=\frac{2\sin (\alpha \pi )}{\pi }\int_{-\infty }^{+\infty
}g_{\lambda ,\psi _{DE}}(x)dx.
\end{equation*}%
The parameter $\tau $ needs to be selected in some way and the analysis is
provided in Section \ref{sectau}. Its introduction is motivated by the fact
that, when moving from $\lambda $ to $\mathcal{L}$, the method (the choice of
$M$, $N$ and $h$) and the error estimates have to be derived by working
uniformly in the interval $[1,+\infty )$ containing $\sigma (\mathcal{L})$.
As in the SE case, the function $g_{\lambda ,\psi _{DE}}(x)$ exhibits a fast
decay for $x\rightarrow \pm \infty $, but the definition of the strip of
analyticity is now much more difficult to handle since everything now
depends on $\lambda $ and $\tau $.


\subsection{Asymptotic behavior of the integrand function}

In order to apply  Theorem \ref{teo2} we need to study $|g_{\lambda ,\psi
_{DE}}(x+i\eta )|$. From (\ref{ide}) we have
\begin{equation*}
\left\vert g_{\lambda ,\psi _{DE}}(x+i\eta )\right\vert =\frac{\pi }{2}%
\lambda ^{-\alpha }\left\vert \frac{\left( \lambda /\tau \exp (\pi \sinh
(x+i\eta ))\right) ^{\alpha }}{1+\lambda /\tau \exp (\pi \sinh (x+i\eta ))}%
\right\vert |\cosh (x+i\eta )|.
\end{equation*}%
After simple manipulations based on standard relations we find
\begin{equation*}
\left\vert \cosh (x+i\eta )\right\vert =\sqrt{\cosh ^{2}x-\sin ^{2}\eta },
\end{equation*}%
and therefore
\begin{equation*}
\left\vert \cosh (x+i\eta )\right\vert \leq \cosh x.
\end{equation*}%
Moreover%
\begin{equation*}
\left\vert \left( \lambda /\tau \exp (\pi \sinh (x+i\eta ))\right) ^{\alpha
}\right\vert =\left( \frac{\lambda }{\tau }\right) ^{\alpha }\left\vert \exp
(\alpha \pi \sinh x\cos \eta ))\right\vert .
\end{equation*}%
In addition, we can bound the denominator using the results given in \cite[%
p. 388]{OK1}, that is,
\begin{equation*}
\left\vert \frac{1}{1+\lambda /\tau \exp (\pi \sinh (x+i\eta ))}\right\vert
\leq \frac{1}{1+\lambda /\tau \exp (\pi \sinh x\cos \eta )\cos (\pi /2\sin
\eta )}.
\end{equation*}%
From the above relations we find
\begin{equation*}
\left\vert g_{\lambda ,\psi _{DE}}(x+i\eta )\right\vert \leq \frac{\pi }{2}%
\lambda ^{-\alpha }\frac{\cosh x}{\cos (\pi /2\sin \eta )}G_{\alpha }(x,\eta
),
\end{equation*}%
where%
\begin{equation*}
G_{\alpha }(x,\eta )=\frac{\left( \lambda /\tau \exp (\pi \sinh x\cos \eta
)\right) ^{\alpha }}{1+\lambda /\tau \exp (\pi \sinh x\cos \eta )}.
\end{equation*}%
Let $x^{\ast }$ be such that
\begin{equation*}
\pi \sinh x^{\ast }\cos \eta =\ln (\tau /\lambda );
\end{equation*}%
we have
\begin{equation}
G_{\alpha }(x,\eta )\leq \left\{
\begin{array}{ll}
(\lambda /\tau )^{\alpha }\exp (\alpha \pi \cos \eta \sinh x), & x\leq
x^{\ast }, \\
(\lambda /\tau )^{\alpha -1}\exp (-(1-\alpha )\pi \cos \eta \sinh x), &
x>x^{\ast }.%
\end{array}%
\right.   \label{31}
\end{equation}


\subsection{Error estimate for the scalar case}

The bound (\ref{31}) implies that
\begin{subequations}
\begin{align*}
\mathcal{N}\left( g_{\lambda ,\psi _{DE}},d\right) &=\lim_{\eta \rightarrow
d^{-}}\left\{ \int_{-\infty }^{+\infty }\left\vert g_{\lambda ,\psi
_{DE}}(x+i\eta )\right\vert dx+\int_{-\infty }^{+\infty }\left\vert
g_{\lambda ,\psi _{DE}}(x-i\eta )\right\vert dx\right\} \\
&\leq \lim_{\eta \rightarrow d^{-}}\pi \lambda ^{-\alpha }\left\{ \frac{1}{%
\cos (\pi /2\sin \eta )}\int_{-\infty }^{+\infty }G_{\alpha }(x,\eta )\cosh
xdx\right\} \\
&\leq \lim_{\eta \rightarrow d^{-}}\frac{\pi \lambda ^{-\alpha }}{\cos (\pi
/2\sin \eta )}\left\{ (\lambda /\tau )^{\alpha }\int_{-\infty }^{x^{\ast
}}\exp (\alpha \pi \cos \eta \sinh x)\cosh x\,dx\right. + \\
&+\left. (\lambda /\tau )^{\alpha -1}\int_{x^{\ast }}^{+\infty }\exp
(-(1-\alpha )\pi \cos \eta \sinh x)\cosh x\,dx\right\} \\
&\leq \frac{1}{\alpha (1-\alpha )}\frac{2}{\cos d\cos (\pi /2\sin d)}%
\lambda ^{-\alpha }.
\end{align*}
\end{subequations}
In addition, assuming $d=d(\lambda ,\tau )<\pi /2$, it can be observed that
\begin{align*}
\int_{-d(\lambda ,\tau )}^{d(\lambda ,\tau )}|g_{\lambda ,\psi
_{DE}}(x+i\eta )|d\eta &\leq \frac{\pi }{2}\lambda ^{-\alpha
}\int_{-d(\lambda ,\tau )}^{d(\lambda ,\tau )}\frac{G_{\alpha }(x,\eta
)\cosh x}{\cos (\pi /2\sin \eta )}d\eta \\
&=\mathcal{O}(1)\quad \text{for }x\rightarrow \pm \infty .
\end{align*}
Using  (\ref{teo1}), for the discretization error we have
\begin{equation*}
\left\vert \int_{-\infty }^{+\infty }g_{\lambda ,\psi
_{DE}}(x)dx-h\sum_{\ell =-\infty }^{+\infty }g_{\lambda ,\psi _{DE}}(\ell
h)\right\vert \leq \xi (d)\frac{1}{\alpha (1-\alpha )}\lambda ^{-\alpha }%
\frac{e^{-\pi d/h}}{2\sinh (\pi d/h)},
\end{equation*}%
where%
\begin{equation}
\xi (d)=\frac{2}{{\cos d\cos (\pi /2\sin d)}}.  \label{CD}
\end{equation}

The remaining task is to estimate the truncation error. Using (\ref{31}) we
obtain
\begin{align*}
h\sum\nolimits_{\ell =-\infty }^{-M-1}\left\vert g_{\lambda ,\psi
_{DE}}(\ell h)\right\vert & \leq \frac{\pi }{2}\tau ^{-\alpha
}h\sum\nolimits_{\ell =-\infty }^{-M-1}\exp (\alpha \pi \sinh (\ell h))\cosh
(\ell h) \\
& \leq \frac{\pi }{2}\tau ^{-\alpha }\int_{-\infty }^{-Mh}\exp (\alpha \pi
\sinh x)\cosh (x)dx \\
& \leq \frac{\tau ^{-\alpha }}{2\alpha }\exp \left( -\alpha \pi \sinh
(Mh)\right) \\
& \leq \frac{\tau ^{-\alpha }}{2\alpha }\exp \left( \frac{\alpha \pi }{2}%
\right) \exp \left( -\frac{\alpha \pi }{2}\exp (Mh)\right) .
\end{align*}%
Similarly,
\begin{align*}
h\sum\nolimits_{\ell =N+1}^{+\infty }\left\vert g_{\lambda ,\psi _{DE}}(\ell
h)\right\vert & \leq \frac{\pi }{2}\lambda ^{-1}\tau ^{1-\alpha
}h\sum\nolimits_{\ell =N+1}^{+\infty }\exp (-(1-\alpha )\pi \sinh (\ell
h))\cosh (\ell h) \\
& \leq \frac{\pi }{2}\lambda ^{-1}\tau ^{1-\alpha }h\int_{Nh}^{+\infty }\exp
(-(1-\alpha )\pi \sinh x)\cosh (x)dx \\
& \leq \frac{\lambda ^{-1}\tau ^{1-\alpha }}{2(1-\alpha )}\exp \left( \frac{%
(1-\alpha )\pi }{2}\right) \exp \left( -\frac{(1-\alpha )\pi }{2}\exp
(Nh)\right) .
\end{align*}

The above results are summarized as follows.

\begin{proposition}
Using the double-exponential transform, for the quadrature error it holds%
\begin{align}
\mathcal{E}_{M,N,h}(g_{\lambda ,\psi _{DE}}) &\leq \frac{1}{\alpha
(1-\alpha )}\xi (d)\lambda ^{-\alpha }\frac{e^{-\pi d/h}}{2\sinh (\pi d/h)}+
\label{de} \\
&\frac{\tau ^{-\alpha }}{2\alpha }\exp \left( \frac{\alpha \pi }{2}\right)
\exp \left( -\frac{\alpha \pi }{2}\exp (Mh)\right) +  \label{te1} \\
&\frac{\lambda ^{-1}\tau ^{1-\alpha }}{2(1-\alpha )}\exp \left( \frac{%
(1-\alpha )\pi }{2}\right) \exp \left( -\frac{(1-\alpha )\pi }{2}\exp
(Nh)\right) ,  \label{te2}
\end{align}%
where $\xi (d)$ is defined by (\ref{CD}).
\end{proposition}

Defining
\begin{equation}
h=\ln \left( \frac{4dn}{\mu }\right) \frac{1}{n},\quad \text{for\quad }n\geq
\frac{\mu e}{4d},\quad \mu =\min (\alpha ,1-\alpha )  \label{h}
\end{equation}%
as in \cite[Theorem 2.14]{OK1}, we first observe that (see (\ref{de}))
\begin{equation}
\frac{\exp \left( -\frac{\pi d}{h}\right) }{2\sinh \left( \frac{\pi d}{h}%
\right) }\leq \frac{1}{1-e^{-\frac{\pi }{2}\mu e}}\exp \left( \frac{-2\pi dn%
}{\ln \left( \frac{4dn}{\mu }\right) }\right) .  \label{kk}
\end{equation}%
Setting $M=N=n$, the choice of $h$ as in (\ref{h}) leads to a truncation
error that decays faster than the discretization one, because for an
arbitrary constant $c$ (see (\ref{te1})-(\ref{te2}))%
\begin{equation*}
\exp \left( -c\exp \left( nh\right) \right) =\exp \left( -\frac{4cdn}{\mu }%
\right) .
\end{equation*}%
As consequence the idea is to assume the discretization error as estimator
for the global quadrature error, that is, using (\ref{de}) and (\ref{kk}),%
\begin{equation}
\mathcal{E}_{n,h}(g_{\lambda ,\psi _{DE}})=\mathcal{E}_{M,N,h}(g_{\lambda
,\psi _{DE}})\approx K_{\alpha }\xi (d)\lambda ^{-\alpha }\exp \left( \frac{%
-2\pi dn}{\ln \left( \frac{4dn}{\mu }\right) }\right) ,  \label{ere}
\end{equation}%
where%
\begin{equation}
K_{\alpha }=\frac{1}{\alpha (1-\alpha )}\frac{1}{1-e^{-\frac{\pi }{2}\mu e}}.
\label{cost}
\end{equation}%
Formula (\ref{ere}) is very similar to the one given in \cite[Theorem 2.14]%
{OK1}, that reads%
\begin{equation}
\mathcal{\hat E}_{M,N,h}(g_{\lambda ,\psi _{DE}})\approx \frac{\tau
^{-\alpha }}{\mu }\alpha (1-\alpha )\left( K_{\alpha }\xi (d)+e^{\frac{\pi }{%
2}\nu }\right) \exp \left( \frac{-2\pi dn}{\ln \left( \frac{4dn}{\mu }%
\right) }\right)   \label{ere2}
\end{equation}%
where $\nu =\max \left( \alpha ,1-\alpha \right) $ and $M=n$, $N=n-\chi$ (or
viceversa depending on $\alpha $), where $\chi>0$ is defined in order to
equalize the contribute of the truncation errors \cite[Theorem 2.11]{OK1}.
The important difference is given by the factor $\lambda ^{-\alpha }$ that
replaces $\tau ^{-\alpha }$, and this is crucial to correctly handle the
case of $\lambda \rightarrow +\infty $. In this situation the error of the
trapezoidal rule goes 0 because $g_{\lambda ,\psi _{DE}}(x)\rightarrow 0$ as
$\lambda \rightarrow +\infty $ (see (\ref{ide})). Anyway, as we shall see, $%
d\rightarrow 0$ as $\lambda \rightarrow +\infty $, so that the exponential
term itself is not able to reproduce this situation. An example is given in
Figure \ref{Figure2} in which we consider $\lambda =10^{12}$ and $\tau =100.$

\begin{figure}[htbp]
  \centering
\includegraphics[width=0.95\textwidth]{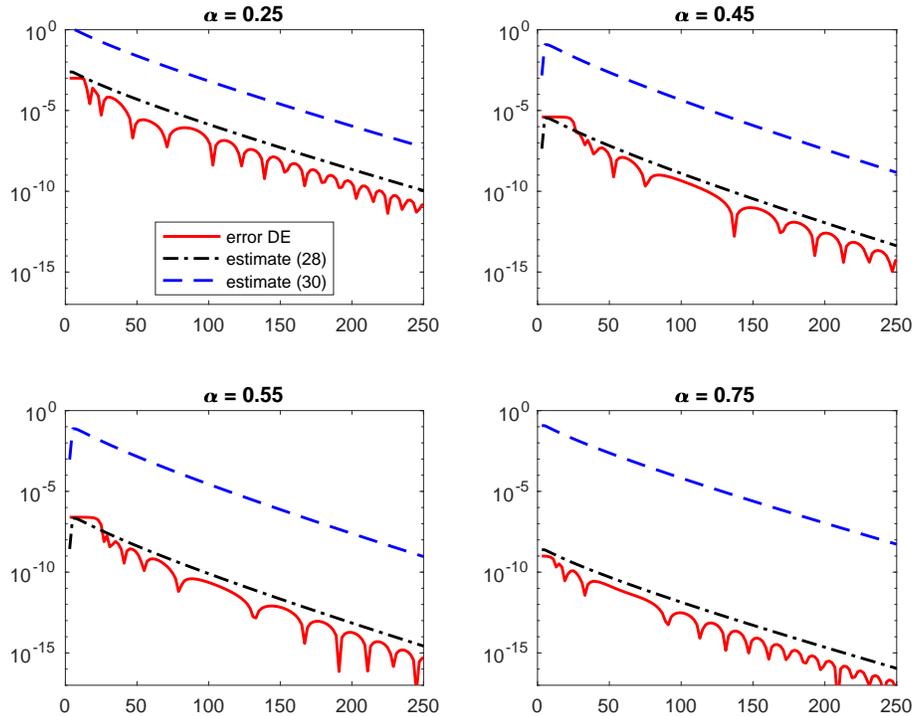}
\caption{Error for the trapezoidal rule applied with the double-exponential transform
(error DE), estimates (\ref{ere}) and (\ref{ere2}) vs the number of inversions, for the
computation of $\lambda ^{-\alpha }$ with $\lambda %
=10^{12}$ and $\tau =100$.}
\label{Figure2}
\end{figure}


\subsection{The poles of the integrand function}

All the analysis presented so far is based on the assumption that the
integrand function%
\begin{equation*}
g_{\lambda ,\psi _{DE}}(x)=\frac{\pi }{2}\tau ^{1-\alpha }\frac{\exp \left(
\alpha \pi \sinh x\right) }{\tau +\lambda \exp \left( \pi \sinh x\right) }%
\cosh x
\end{equation*}%
is analytic in the strip $\mathcal{D}_{d}$, for a certain $d=d(\lambda ,\tau
)$. Therefore we have to study the poles of this function, that is, we have
to study the equation%
\begin{equation*}
\tau +\lambda \exp \left( \pi \sinh x\right) =0.
\end{equation*}%
We have%
\begin{align*}
\exp \left( \pi \sinh x\right)  &=\frac{\tau }{\lambda }e^{i\pi }, \\
\sinh x &=\frac{1}{\pi }\ln \frac{\tau }{\lambda }+i(2k+1),\quad k\in
\mathbb{Z}.
\end{align*}%
By solving the above equation for each $k$, we obtain the complete set of
poles. Assuming to work with the principal value of the logarithm and taking
$k=0$, we obtain the poles closest to the real axis $x_{0}$ and its
conjugate $\overline{x_{0}}$, where%
\begin{align}
x_{0} &=\sinh ^{-1}\left( \frac{1}{\pi }\ln \frac{\tau }{\lambda }+i\right)
\notag \\
&=\ln \left( \frac{1}{\pi }\ln \frac{\tau }{\lambda }+i+\sqrt{\left( \frac{1%
}{\pi }\ln \frac{\tau }{\lambda }\right) ^{2}+2i\frac{1}{\pi }\ln \frac{\tau
}{\lambda }}\right) .  \label{x0}
\end{align}%
In order to apply the bound on the strip we have to define%
\begin{equation}
d=d(\lambda ,\tau )=r\func{Im}x_{0},\quad 0<r<1.  \label{d}
\end{equation}%
The introduction of the factor $r$ is necessary to avoid ${\xi }{%
(d)\rightarrow +\infty }$ as $\func{Im}x_{0}\rightarrow \pi /2$, which
verifies for $\lambda \rightarrow \tau $ (see (\ref{CD})).


\subsubsection{Asymptotic behaviors}

Setting
\begin{equation*}
s=\frac{1}{\pi }\ln \frac{\lambda }{\tau },
\end{equation*}%
we have
\begin{equation*}
\frac{1}{\pi }\ln \frac{\tau }{\lambda }=-s,
\end{equation*}%
and therefore we can write (\ref{x0}) as
\begin{equation*}
x_{0}=\ln \left( s\left( -1+\frac{i}{s}+\sqrt{1-\frac{2i}{s}}\right) \right).
\end{equation*}%
Assuming $\lambda \gg \tau $, that is, $s\gg 1$, and using%
\begin{equation}
\sqrt{1-x}\approx 1-\frac{1}{2}x-\frac{1}{8}x^{2}-\frac{1}{16}x^{3},\quad
x\approx 0,  \label{sv4}
\end{equation}%
we obtain%
\begin{equation*}
\sqrt{1-\frac{2i}{s}}\approx 1-\frac{i}{s}+\frac{1}{2s^{2}}+\frac{i}{2s^{3}}.
\end{equation*}%
Using also $\ln (1+x)\approx x$,%
\begin{align*}
x_{0} &\approx \ln \left( s\left( -1+\frac{i}{s}+1-\frac{i}{s}+\frac{1}{%
2s^{2}}+\frac{i}{2s^{3}}\right) \right)  \\
&=\ln \left( s\left( \frac{1}{2s^{2}}+\frac{i}{2s^{3}}\right) \right)  \\
&=\ln \left( \frac{1}{2s}\right) +\ln \left( 1+\frac{i}{s}\right)  \\
&\approx \ln \left( \frac{1}{2s}\right) +\frac{i}{s}.
\end{align*}%
Therefore, for $\lambda \gg \tau $,%
\begin{equation}
\func{Im}x_{0}\approx \frac{1}{s}=\frac{\pi }{\ln \frac{\lambda }{\tau }}.
\label{app1}
\end{equation}

Assume now $\lambda =1$ and $\tau \gg 1$. By (\ref{x0}) we have%
\begin{equation*}
x_{0}=\ln \left( \frac{1}{\pi }\ln \tau +i+\sqrt{\left( \frac{1}{\pi }\ln
\tau \right) ^{2}+2i\frac{1}{\pi }\ln \tau }\right) .
\end{equation*}%
Setting%
\begin{equation*}
s=\frac{1}{\pi }\ln \tau ,
\end{equation*}%
we have%
\begin{align*}
x_{0} &=\ln \left( s\left( 1+\frac{i}{s}+\sqrt{1+\frac{2i}{s}}\right)
\right) \\
&\approx \ln \left( s\left( 1+\frac{i}{s}+1+\frac{i}{s}\right) \right) \\
&=\ln \left( 2s\left( 1+\frac{i}{s}\right) \right) \\
&\approx \ln \left( 2s\right) +\frac{i}{s},
\end{align*}%
that finally leads to%
\begin{equation}
\func{Im}x_{0}\approx \frac{1}{s}=\frac{\pi }{\ln \tau }.  \label{app2}
\end{equation}


\subsection{The minimax problem}

Let us define the function%
\begin{equation*}
\varphi (\lambda ,\tau )=\xi (d)\lambda ^{-\alpha }\exp \left( \frac{-2\pi dn%
}{\ln \left( \frac{4dn}{\mu }\right) }\right) ,\quad d=d(\lambda ,\tau ),
\end{equation*}%
representing the $(\lambda ,\tau )$-dependent factor of the error estimate
given by (\ref{ere}), that is,%
\begin{equation*}
\mathcal{E}_{n,h}(g_{\lambda ,\psi _{DE}})\approx K_{\alpha }\varphi
(\lambda ,\tau ),
\end{equation*}%
where $K_{\alpha }$ is defined by (\ref{cost}). Since our aim is to work
with a self-adjoint operator with spectrum contained in $[1, +\infty )$ the
problem consists in defining properly the parameter $\tau.$ This can be
done by solving%
\begin{equation}
\min_{\tau \geq 1}\max_{\lambda \geq 1}\varphi (\lambda ,\tau ).
\label{minmax}
\end{equation}%
As for the true error, experimentally one observes that $\tau $ must be
taken much greater than 1, independently of $\alpha $. Therefore, from now
on the analysis will be based on the assumption $\tau \gg 1$. Regarding the
function $\varphi (\lambda ,\tau )$, by taking $d=d(\lambda ,\tau )$ as in (\ref{d}) 
and $n$ sufficiently large, again, one experimentally observes that
with respect to $\lambda $ the function initially decreases, reaches a local
minimum (for $\lambda =\tau $ in which $d=r\pi /2$), then a local maximum
(much greater than $\tau $), and finally goes to 0 for $\lambda \rightarrow
+\infty $ (see Figure \ref{Figure3}). In this view, denoting by $\overline{%
\lambda }$ the local maximum, for $n$ sufficiently large the problem (\ref{minmax}) reduces to the solution of%
\begin{equation}
\varphi (1,\tau )=\varphi (\overline{\lambda },\tau ).  \label{eqt}
\end{equation}


\subsubsection{Evaluating the local maximum}

Since $0<d\leq r\pi /2$, $0<r<1$, we have%
\begin{equation*}
0<C\leq \cos d\cos \left( \frac{\pi }{2}\sin d\right) <1,
\end{equation*}%
where $C$ is a constant depending on $r$. Therefore by (\ref{CD}),%
\begin{equation*}
2 < \xi (d)\leq \frac{2}{C},
\end{equation*}%
so that we neglect the contribute of this function in what follows.

Since the maximum is seen to be much larger than $\tau $, we consider the
approximation (\ref{app1}). Therefore we have to solve%
\begin{equation*}
\frac{d}{d\lambda }\lambda ^{-\alpha }\exp \left( -\frac{2\pi r\frac{\pi }{%
\ln \frac{\lambda }{\tau }}n}{\ln \left( \frac{4}{\mu }nr\frac{\pi }{\ln
\frac{\lambda }{\tau }}\right) }\right) =0,
\end{equation*}%
that, after some manipulation leads to%
\begin{equation*}
\frac{d}{d\lambda }\lambda ^{-\alpha }\exp \left( -\frac{c_{1}n}{\ln \frac{%
\lambda }{\tau }\,q(\lambda )}\right) =0,
\end{equation*}%
where%
\begin{equation}
c_{1}=2\pi ^{2}r,\quad q(\lambda )=\ln \left( c_{2}n\right) -\ln \left( \ln
\frac{\lambda }{\tau }\right) ,\quad c_{2}=\frac{4}{\mu }\pi r.  \label{c12}
\end{equation}

We find the equation%
\begin{equation*}
-\alpha \lambda ^{-1}-\frac{d}{d\lambda }\left( \frac{c_{1}n}{\ln \frac{%
\lambda }{\tau } \, q(\lambda) }\right) =0,
\end{equation*}%
and since%
\begin{align*}
\frac{d}{d\lambda }\left( \frac{c_{1}n}{\ln \frac{\lambda }{\tau } \,
q(\lambda) }\right) &=\frac{c_{1}n}{\lambda }\frac{1-q(\lambda )}{\left(
\ln \frac{\lambda }{\tau }\right) ^{2}q(\lambda )^{2}},
\end{align*}%
we finally have to solve%
\begin{equation}
\alpha +c_{1}n\frac{1-q(\lambda )}{\left( \ln \frac{\lambda }{\tau }\right)
^{2}q(\lambda )^{2}}=0.  \label{eqs}
\end{equation}

For large $n$ we have%
\begin{align*}
q(\lambda ) &\approx \ln \left( c_{2}n\right) , \\
\frac{q(\lambda )-1}{q(\lambda )^{2}} &\approx \frac{1}{q(\lambda )}\approx
\frac{1}{\ln \left( c_{2}n\right) },
\end{align*}%
so that the solution of (\ref{eqs}) can be approximated by%
\begin{equation}
\lambda ^{\ast }=\tau \exp \left( \sqrt{\frac{c_{1}n}{\alpha \ln \left(
c_{2}n\right) }}\right) .  \label{m2}
\end{equation}%
For any given $\tau \geq 1,$ it can be observed experimentally that $\lambda
^{\ast }$ is a very good approximation of the local maximum (see  Figure \ref{Figure3}).

We also remark that the assumption on $n$ stated in (\ref{h}), that leads to
the error estimate (\ref{ere}), is automatically fulfilled for $\lambda
=\lambda ^{\ast }$, at least for $\alpha $ not too small. Indeed, using (\ref{d}) and (\ref{app1}) we first observe that
\begin{equation}
d(\lambda ^{\ast },\tau )\approx \frac{r\pi }{\ln \frac{\lambda ^{\ast }}{%
\tau }}=r\pi \sqrt{\frac{\alpha \ln \left( c_{2}n\right) }{c_{1}n}}.
\label{dbd}
\end{equation}%
Then by (\ref{c12}), using $\mu \leq 1/2$ and assuming for instance $0.9<r<1,$
after some simple computation we find
\begin{align*}
\frac{\mu e}{4d(\lambda ^{\ast },\tau )} &\approx \frac{\mu e}{4\pi r}\sqrt{%
\frac{c_{1}n}{\alpha \ln \left( c_{2}n\right) }} \\
&\leq \frac{1}{3}\sqrt{\frac{n}{\alpha }}.
\end{align*}%
Therefore the condition (\ref{h}) holds true for $n\geq 1/(9\alpha).$


\subsubsection{The error at the local maximum}

By (\ref{dbd}) clearly   $d(\lambda^{\ast },\tau )\rightarrow 0$ for $%
n\rightarrow +\infty $, and therefore from (\ref{CD}) we deduce that $\xi
(d(\lambda ^{\ast },\tau ))\rightarrow 2$ for $n\rightarrow +\infty $. As
consequence%
\begin{equation*}
\varphi (\lambda ^{\ast },\tau )\approx 2\left( \lambda ^{\ast }\right)
^{-\alpha }\exp \left( \frac{-2\pi d(\lambda ^{\ast },\tau )n}{\ln \left(
\frac{4d(\lambda ^{\ast },\tau )n}{\mu }\right) }\right) .
\end{equation*}%
By defining
\begin{equation}
s_{n}=\sqrt{\frac{c_{1}n}{\ln \left( c_{2}n\right) }},  \label{sn}
\end{equation}%
from (\ref{m2}) and (\ref{dbd}) we have%
\begin{align*}
\lambda ^{\ast } &=\tau \exp \left( \frac{s_{n}}{\sqrt{\alpha }}\right) , \\
d(\lambda ^{\ast },\tau ) &\approx \frac{\sqrt{\alpha }r\pi }{s_{n}},
\end{align*}%
and hence, after some computation%
\begin{align*}
\left( \lambda ^{\ast }\right) ^{-\alpha }\exp \left( \frac{-2\pi d(\lambda
^{\ast },\tau )n}{\ln \left( \frac{4d(\lambda ^{\ast },\tau )n}{\mu }\right)
}\right)  &\approx \tau ^{-\alpha }\exp \left( -\sqrt{\alpha }s_{n}\right)
\exp \left( \frac{-2\pi \frac{\sqrt{\alpha }r\pi }{s_{n}}n}{\ln \left( \frac{%
4\frac{\sqrt{\alpha }r\pi }{s_{n}}n}{\mu }\right) }\right)  \\
&=\tau ^{-\alpha }\exp \left( -\sqrt{\alpha }\left( s_{n}+\frac{c_{1}n}{%
s_{n}\ln \left( c_{2}n\frac{\sqrt{\alpha }}{s_{n}}\right) }\right) \right) .
\end{align*}%
By (\ref{sn}) we have%
\begin{align*}
s_{n}+\frac{c_{1}n}{s_{n}\ln \left( c_{2}n\frac{\sqrt{\alpha }}{s_{n}}%
\right) } &=\sqrt{\frac{c_{1}n}{\ln \left( c_{2}n\right) }}+\frac{c_{1}n}{%
\sqrt{\frac{c_{1}n}{\ln \left( c_{2}n\right) }}\ln \left( c_{2}n\frac{\sqrt{%
\alpha }}{s_{n}}\right) } \\
&=\sqrt{\frac{c_{1}n}{\ln \left( c_{2}n\right) }}\left( 1+\frac{\ln \left(
c_{2}n\right) }{\ln \left( c_{2}n\frac{\sqrt{\alpha }}{s_{n}}\right) }%
\right)  \\
&\approx 3\sqrt{\frac{c_{1}n}{\ln \left( c_{2}n\right) }},
\end{align*}%
because
\begin{equation*}
\frac{\ln \left( c_{2}n\right) }{\ln \left( c_{2}n\frac{\sqrt{\alpha }}{s_{n}%
}\right) }\rightarrow 2\text{\quad for\quad }n\rightarrow +\infty .
\end{equation*}%
Joining the above approximations we finally obtain%
\begin{align}
\varphi (\lambda ^{\ast },\tau ) &\approx 2\tau ^{-\alpha }\exp \left( -3%
\sqrt{\alpha }\sqrt{\frac{c_{1}n}{\ln \left( c_{2}n\right) }}\right)   \notag
\\
&=2\tau ^{-\alpha }\exp \left( -3\sqrt{\alpha }s_{n}\right) .  \label{errm2}
\end{align}


\subsubsection{Error at $\lambda =1$}

By (\ref{app2}), that is,
\begin{equation*}
d(1,\tau )\approx r\frac{\pi }{\ln \tau },\quad \tau \gg 1,
\end{equation*}%
we have again $\xi (d(1,\tau ))\approx 2$ and therefore
\begin{equation*}
\varphi (1,\tau )\approx 2\exp \left( \frac{-2\pi d(1,\tau )n}{\ln \left(
\frac{4d(1,\tau )n}{\mu }\right) }\right) .
\end{equation*}%
Using (\ref{c12}) we find%
\begin{align}
\varphi (1,\tau ) &\approx 2\exp \left( -\frac{2\pi r\frac{\pi }{\ln \tau }n%
}{\ln \left( \frac{4}{\mu }nr\frac{\pi }{\ln \tau }\right) }\right)  \notag
\\
&=2\exp \left( -\frac{c_{1}n}{\ln \tau \left( \ln \left( c_{2}n\right) -\ln
\left( \ln \tau \right) \right) }\right)  \notag \\
&\approx 2\exp \left( -\frac{c_{1}n}{\ln \tau \ln \left( c_{2}n\right) }%
\right)  \notag \\
&=2\exp \left( -\frac{s_{n}^{2}}{\ln \tau }\right) .  \label{t}
\end{align}


\subsection{Approximating the optimal value for $\protect\tau $}

\label{sectau}

We need to solve (\ref{eqt}). Using the approximations (\ref{t}) and (\ref{errm2}) we impose%
\begin{align*}
\exp \left( -\frac{s_{n}^{2}}{\ln \tau }\right) &=\tau ^{-\alpha }\exp
\left( -3\sqrt{\alpha }s_{n}\right) \\
&=\exp \left( -3\sqrt{\alpha }s_{n}-\alpha \ln \tau \right) ,
\end{align*}%
that is,%
\begin{equation*}
-\frac{s_{n}^{2}}{\ln \tau }=-3\sqrt{\alpha }s_{n}-\alpha \ln \tau .
\end{equation*}%
Solving the above equation we find%
\begin{align*}
\ln \tau &=\frac{\left( -3+\sqrt{13}\right) \sqrt{\alpha }s_{n}}{2\alpha }
\\
&\approx 0.3\frac{s_{n}}{\sqrt{\alpha }},
\end{align*}%
so that%
\begin{equation}
\tau ^{\ast }=\exp \left( 0.3\frac{s_{n}}{\sqrt{\alpha }}\right)  \label{tau}
\end{equation}%
represents an approximate solution of (\ref{eqt}).

\begin{figure}[htbp]
  \centering
\includegraphics[width=0.95\textwidth]{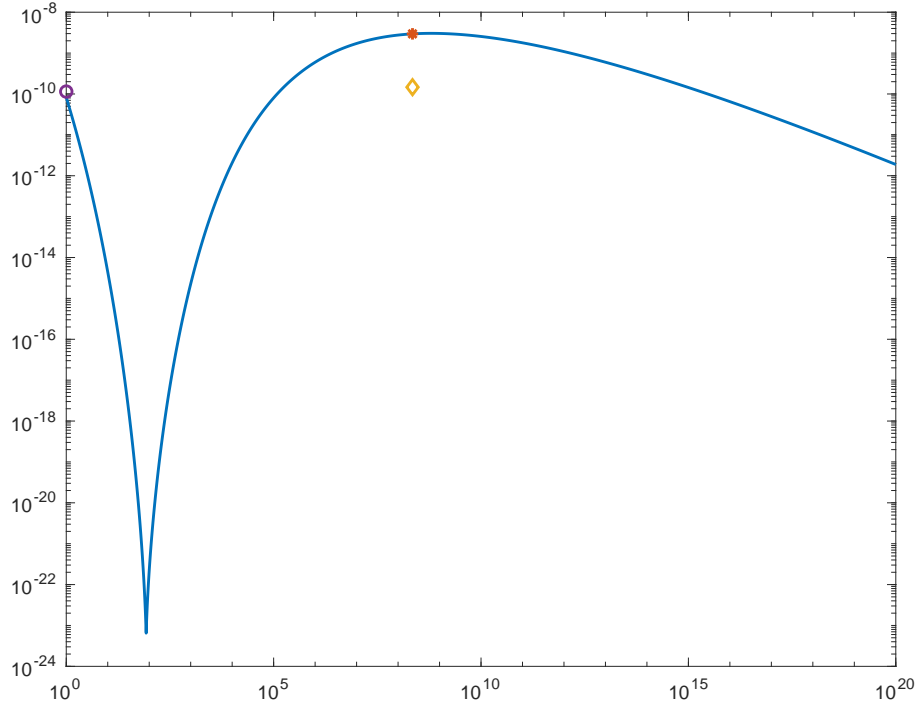}
\caption{Plot of the function $\varphi (\lambda,\tau ^{\ast })$ for $n=40$ and $\alpha =1/2$. The asterisk
represents the approximation of the local maximum given by (\ref{m2}), that is, the point 
$\left( \lambda ^{\ast },\varphi (\lambda ^{\ast },\tau ^{\ast })\right) $. The diamond
represents the approximation of $\varphi (\lambda ^{\ast },%
\tau ^{\ast })$ stated in (\ref{errm2}). Finally the circle
is the approximation of $\varphi (1,\tau ^{\ast })$ given in
(\ref{t}).}
\label{Figure3}
\end{figure}
In Figure \ref{Figure3} we plot the function $\varphi (\lambda ,\tau ^{\ast
})$ for $\lambda \in \lbrack 1,10^{20}]$, in an example in which $n=40$, $%
\alpha =1/2$, and $\tau ^{\ast }\cong 84.4$ defined by (\ref{tau}). Moreover
we show the results of the approximations (\ref{m2}), (\ref{errm2}) and (\ref{t}), for $\tau =\tau ^{\ast }$. Clearly the ideal situation would be to
have $\tau ^{\ast }$ such that $\varphi (1,\tau ^{\ast })=\varphi (\overline{%
\lambda },\tau ^{\ast })$, but notwithstanding all the approximations used,
the results are fairly good and allow to have a simple expression for $\tau
^{\ast }$.

By using (\ref{tau}) in (\ref{errm2}) we obtain%
\begin{equation}
\varphi (\lambda ^{\ast },\tau ^{\ast })\approx 2\exp \left( -3.3\sqrt{%
\alpha }\sqrt{\frac{c_{1}n}{\ln \left( c_{2}n\right) }}\right) .  \label{gtt}
\end{equation}%
Remembering that%
\begin{equation*}
E_{n,h}(\mathcal{L})\leq 2\frac{\sin \left( \alpha \pi \right) }{\pi }%
\max_{\lambda \geq 1}\mathcal{E}_{n,h}(g_{\mathcal{\lambda },\psi _{DE}}),
\end{equation*}%
using (\ref{gtt}) we finally obtain the error estimate%
\begin{equation}
E_{n,h}(\mathcal{L})\approx \overline{K}_{\alpha }\exp \left( -3.3\sqrt{%
\alpha }\sqrt{\frac{c_{1}n}{\ln \left( c_{2}n\right) }}\right) ,
\label{fest}
\end{equation}%
where%
\begin{align*}
\overline{K}_{\alpha } &=4\frac{\sin \left( \alpha \pi \right) }{\pi }%
K_{\alpha } \\
&=4\frac{\sin \left( \alpha \pi \right) }{\pi }\frac{1}{\alpha (1-\alpha )}%
\frac{1}{1-e^{-\frac{\pi }{2}\mu e}}.
\end{align*}%
In Figure \ref{Figure4} we show the behavior of the method for the
computation of $\mathcal{L}^{-\alpha }$, with $\mathcal{L}$ defined in (\ref{scex}), together with the estimate (\ref{fest}). For comparison, in the
same pictures we also plot the error of the SE approach. As mentioned in the
introduction, the DE approach appears to be faster for $1/2 \le \alpha <1.$

\begin{figure}[htbp]
  \centering
\includegraphics[width=0.95\textwidth]{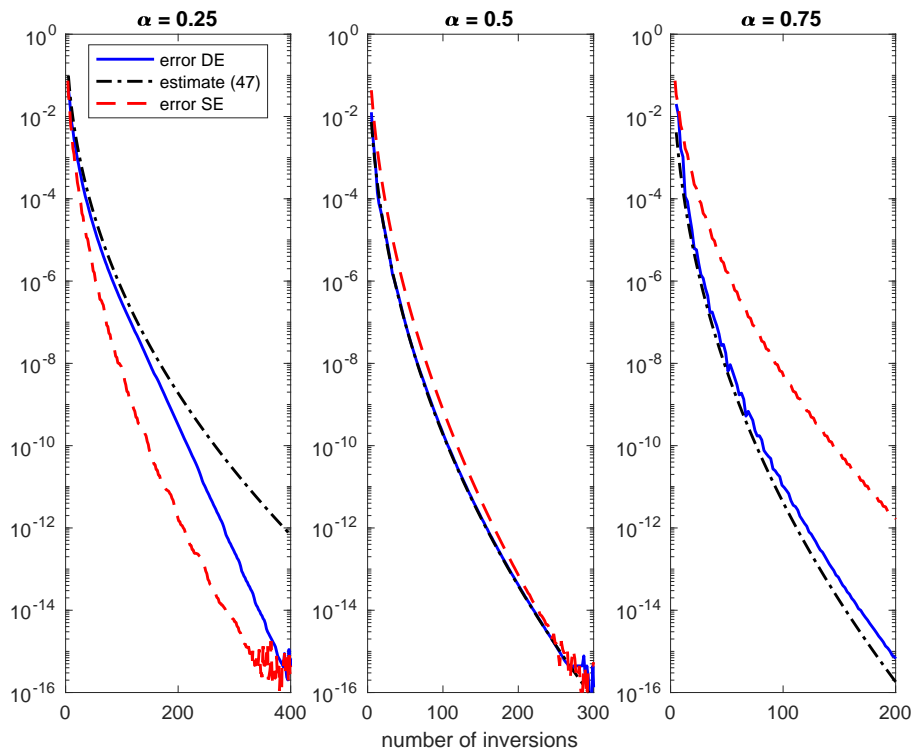}
\caption{Error for the trapezoidal rule applied with the double-exponential
transform (error DE), with the single-exponential transform (error SE) and
error estimate given by (\protect\ref{fest}).}
\label{Figure4}
\end{figure}

\section{Conclusions}

\label{sec5}

In this work we have analyzed the behavior of the trapezoidal rule for the
computation of $\mathcal{L}^{-\alpha }$, in connection with the single and
the double-exponential transformations. All the analysis has been based on
the assumption of $\mathcal{L}$ unbounded, so that the results can be
applied even to discrete operators, with spectrum arbitrarily large, without
the need to know its amplitude, that is, the largest eigenvalue. In
particular we have revised the analysis for the single-exponential transform
and we have introduced new error estimates for the scalar and the operator
case for the double-exponential transform. The sharp estimate obtained for
the scalar case has been fundamental for the proper selection of the
parameter $\tau $ that is necessary to obtain good results also for the
operator case.

\end{document}